\documentclass[a4paper]{article}
\usepackage{amsfonts}
\usepackage{amssymb}
\usepackage{amsthm}
\usepackage{mathrsfs}
\usepackage{amsmath}

\newtheorem{theorem}{Theorem}[section]
\newtheorem{lemma}[theorem]{Lemma}
\newtheorem{corollary}[theorem]{Corollary}
\newtheorem{proposition}[theorem]{Proposition}

\theoremstyle{definition}

\theoremstyle{remark}
\newtheorem{remark}[theorem]{Remark}
\newtheorem{example}[theorem]{Example}
\newtheorem*{acknowledgements}{Acknowledgments}

\newcommand{\scal}[1]{\langle#1\rangle}
\newcommand{\ran}{\mathrm{ran}\,}
\newcommand{\dom}{\mathrm{dom}\,}

\newcommand{\img}{\mathrm{Im}\,}

\newcommand{\C}{\mathbb{C}}
\newcommand{\R}{\mathbb{R}}
\newcommand{\Z}{\mathbb{Z}}

\title{On operators which are adjoint to each other}
\author{Dan
POPOVICI\thanks{ Department of Mathematics, West University of
Timi\c soara, Bd. Vasile P\^arvan nr. 4, RO-300223 Timi\c soara,
Romania, \texttt{popovici@math.uvt.ro}}\qquad Zolt\'an SEBESTY\'EN\thanks{Department of Applied Analysis,
Lor\'and E\"otv\"os University, P\'azm\'any P\'eter s\'et\'any 1/C,
H-1117 Budapest, Hungary, \texttt{sebesty@cs.elte.hu}}}

\date{\small\it Dedicated to the memory of B\'ela Sz\"okefalvi-Nagy}

\begin{document}

\maketitle


\begin{abstract}
Given two linear operators $S$ and $T$ acting between Hilbert spaces $\mathscr{H}$ and $\mathscr{K}$, respectively $\mathscr{K}$ and $\mathscr{H}$ which satisfy the relation
\begin{equation*}
\langle Sh, k\rangle=\langle h, Tk\rangle, \quad h\in\dom S, \ k\in\dom T,
\end{equation*}
i.e., according to the classical terminology of M.H. Stone, which are adjoint to each other, we provide necessary and sufficient conditions in order to ensure the equality between the closure of $S$ and the adjoint of $T.$ A central role in our approach is played by the range of the operator matrix $M_{S, T}=\begin{pmatrix} 1_{\dom S} & -T\\ S & 1_{\dom T} \end{pmatrix}.$ We obtain, as consequences, several results characterizing skewadjointness, selfadjointness and essential selfadjointness. We improve, in particular, the celebrated selfadjointness criterion of J. von Neumann.

\medskip
\noindent\textsc{Keywords}: unbounded operator, operators which are adjoint to each other, symmetric, skewadjoint, selfadjoint, essentially selfadjoint, closable

\end{abstract}

\maketitle

\section{Introduction}

The abstract spectral theory for unbounded selfadjoint operators has been initiated by J. von Neumann in \cite{neu30}, his work being motivated by its applications in quantum mechanics, a branch of physics developed in mid-twenties by P. Dirac, W. Heisenberg and E. Schr\"odinger.
In quantum mechanics the Hamiltonian $H$ of a system is the operator corresponding to the total energy of that system. The domain of this differential operator is sometimes unclear. It is usually not difficult to restrict its definition to some dense subspace of regular functions on which it becomes a symmetric operator. In our approach, a \emph{symmetric operator} $S$ is just a linear operator which is defined on a Hilbert space $\mathscr{H}$ and verifies the identity
\begin{equation}\label{eq1}
\scal{Sh,h'}=\scal{h,Sh'}
\end{equation} 
for every $h$ and $h'$ in the domain $\dom S$ of $S$. In other words, contrary to the usual terminology, we do not assume that a symmetric operator is densely defined.

The dynamics of a quantum system must be governed by a continuous one-parameter group of unitary operators. The infinitesimal generator of such a group is, according to Stone's theorem \cite{sto32}, selfadjoint i.e., it is densely defined and coincides with its adjoint. The \emph{adjoint operator} $S^*$ of a linear, densely defined, operator $S$ (which acts between Hilbert spaces $\mathscr{H}$ and $\mathscr{K}$) is defined by the formulas

\begin{itemize}
\item $\dom S^*:=\{ k\in \mathscr{K} : \langle Sh, k\rangle=\langle h, h'\rangle$ for a certain $h'\in\mathscr{H}$ and for each $h\in\dom S\};$
\item $S^*k:=h',$ where $h'$ is (uniquely) determined by $k$ and the (scalar product) equality above. 
\end{itemize}
Moreover, we have the following identities (see \cite{seb83} for the first and the third one): 
\begin{itemize}
\item  the domain of $S^*$: 
\begin{equation*} \dom S^*=\{ k\in\mathscr{K}: \sup\{ |\langle Sh, k\rangle| : h\in \dom S, \| h\|\le 1\}<\infty\};\end{equation*}
\item the kernel of $S^*:$ 
\begin{equation*} \ker S^*=\{ \ran S\}^\perp, \mbox{ i.e., the orthocomplement of the range } \ran S \mbox{ of } S;\end{equation*}
\item the range of $S^*:$
 \begin{equation}\label{eq0} \ran S^*=\{ h\in\mathscr{H}: \sup\{ |\langle h, h'\rangle| : h'\in \dom S, \| Sh'\|\le 1\}<\infty\}.\end{equation}
\end{itemize}
In the most general context it is sometimes necessary to specify domains of selfadjointness for physical observables such as position, momentum or spin. To this aim one needs to find or even to construct explicitly the selfadjoint extension of an appropriate symmetric operator. The mathematician has to study if this operator has such extensions, while the physicist has to choose among these extensions the one which is the most adequate for the system. In some particular situations the Hamiltonian of the system is essentially selfadjoint, i.e., it has a unique selfadjoint extension. As noted by T. Kato \cite{kat51}, this is the case for the so-called Schr\"odinger Hamiltonian operator of every system composed of a finite number of particles interacting with each other through a potential energy. In abstract settings, a linear, densely defined and closable operator $S$ between $\mathscr{H}$ and $\mathscr{K}$ is called \emph{essentially selfadjoint} if its closure $\bar S$ is selfadjoint. 

One can usually verify selfadjointness by computing the ranges of certain associated operators. We should mention in this sense that, according to the criterion of J. von Neumann \cite{neu30}, a densely defined symmetric operator $S$ acting on a complex Hilbert space $\mathscr{H}$ is selfadjoint if and only if the ranges of the operators $S\pm i$ equal $\mathscr{H}$. Similarly, $S$ is essentially selfadjoint if and only if the ranges of the operators $S\pm i$ are dense in $\mathscr{H}$. 
For basic facts about unbounded operators we refer to the Chapter 8 of the famous book \emph{Functional Analysis} by Frigyes Riesz and B\'ela Sz\"okefalvi-Nagy \cite{risz}.

The following result, stated by R. Arens \cite{are61} for the generalized case of multivalued operators (linear relations), is often useful in studying selfadjontness:
\begin{proposition}\label{p1.1}
Let $S$ and $T$ be two linear relations between linear spaces $\mathscr{X}$ and $\mathscr{Y}$ such that $S\subseteq T$. Then $S=T$ if and only if $S$ and $T$ have equal kernels and equal ranges.
\end{proposition} 

We follow the classical terminology of M. H. Stone \cite{stone} in order to extend \eqref{eq1} for the case of two linear operators $S$ and $T$ acting between Hilbert spaces $\mathscr{H}$ and $\mathscr{K}$, respectively $\mathscr{K}$ and $\mathscr{H}$: $S$ and $T$ are said to be \textit{adjoint to each other}, in symbols $S\wedge T$, if
\begin{equation*}
\scal{Sh,k}=\scal{h,Tk}
\end{equation*}
for every $h\in\dom S$ and $k\in\dom T$. In our approach the range of the operator matrix $M_{S,T}=\begin{pmatrix}1_{\dom S} & -T\\ S & 1_{\dom T}\end{pmatrix}$ is playing a central role. We firstly show that $S=T^*$ and $S^*=T$ if and only if $S\wedge T$ and $\ran(M_{S,T})=\mathscr{H}\times\mathscr{K}$ if and only if $S, T$ are closed and densely defined, $S\wedge T, \ran (1+ST)=\mathscr{K}$ and $\ran (1+TS)=\mathscr{H}.$ We obtain, as a consequence, that for a given operator $S$ on $\mathscr{H}$ the following conditions are equivalent:
\begin{itemize}
\item $S$ is selfadjoint;
\item $S\wedge S$ and $\ran M_{S, S}=\mathscr{H}$;
\item $S$ is closed and densely defined, $S\wedge S$ and $\ran (1+S^2)=\mathscr{H}.$
\end{itemize}
In the case when the Hilbert space $\mathscr{H}$ is complex we revise the von Neumann theorem (indicated above) without the assumption that the operator in discussion is ``a priori'' densely defined. More exactly, $S$ is selfadjoint if and only if $S\wedge S, \ran (\lambda i+S)=\mathscr{H}$ and $\overline{\ran (\lambda i-S)}=\mathscr{H}$ for a certain (and also for all) $\lambda\in\mathbb{R}, \lambda\ne 0.$ Similar characterizations are obtained for skewadjoint operators (i.e., operators $S$ which verify the relation $S^*+S=0$). Provided that $S$ and $T$ are adjoint to each other and $M_{S, T}$ has dense range, we prove, in the final section, that $T$ has dense range if and only if $S$ is closable and, if either of these statements is satisfied, the closure of $S$ equals the adjoint of $T.$ As applications we obtain new conditions which are equivalent to essential selfadjointness. To be more precise, a given linear operator $S$ on $\mathscr{H}$ is essentially selfadjoint if and only if $S$ is densely defined/closable, $S\wedge S$ and the range of $M_{S, S}$ is dense. This last condition can be also replaced by $-\lambda^2\notin\sigma_p(S^{*2})$ for a certain (and also for all) $\lambda\in\mathbb{R}, \lambda\ne 0.$ Also, using another approach based on the Arens result (Proposition \ref{p1.1}), we prove that $S$ is essentially selfadjoint if and only if $S$ is closable, $S\wedge S, \{ \ran S\}^\perp=\ker \bar{S}$ and $\{ h\in\mathscr{H} | \sup \{ | \langle h, h'\rangle : h'\in \dom S$ and $\| Sh'\|\le 1\}<\infty\}=\ran \bar{S}.$ Similar characterizations for (essentially) selfadjoint operators were obtained by Z. Sebesty\'en and Zs. Tarcsay in \cite{stess, stsa}.

We should also note that, in our approach (i.e., $S$ and $T$ are adjoint to each other), the matrices 
$M_{-S,T}$ and $M_{S,-T}$ are both symmetric. In some earlier work the role of these matrices has been played by some $2\times 2$ symmetric off-diagonal matrices (see, for example, \cite{hkm, HM, M, SS, T}).

Corresponding results also hold true for the generalized case of linear relations. This case will be treated in another paper.  

\section{The operator matrix $M_{S,T}$}

Let $\mathscr{H}$ and $\mathscr{K}$ be two Hilbert spaces. The \emph{graph of a linear operator $S$} between Hilbert spaces $\mathscr{H}$ and $\mathscr{K}$ is, as usual, the linear subspace 
\begin{equation*}
G(S):=\{ (h, Sh) : h\in \dom S\}\subseteq \mathscr{H}\times \mathscr{K}. 
\end{equation*}
The \emph{flip operator}, which has its role in computations with linear operator graphs, is defined by the formula
\begin{equation*}
\mathscr{H}\times\mathscr{K}\ni (h, k)\mapsto F_{\mathscr{H}, \mathscr{K}}(h, k):=(k, -h)\in \mathscr{K}\times\mathscr{H}.
\end{equation*} 
It is not hard to check that $F_{\mathscr{H}, \mathscr{K}}$ is a unitary operator and $F^*_{\mathscr{H}, \mathscr{K}}=-F_{\mathscr{K}, \mathscr{H}}.$

It is well-known (cf., for example, \cite[page 304]{risz}), under the assumption that  $S$ is densely defined, that the equality
\begin{equation}\label{eq2}
\overline{G(S)}\oplus F^*_{\mathscr{H}, \mathscr{K}}G(S^*)=\mathscr{H}\times \mathscr{K}
\end{equation}
holds true (the symbol ``$\oplus$'' denotes an orthogonal sum).

The first result computes the orthocomplement of the range of the operator matrix $M_{S, T}.$

\begin{proposition}\label{p1}
Let $S$ and $T$ be two given linear operators acting between $\mathscr{H}$ and $\mathscr{K}$, respectively $\mathscr{K}$ and $\mathscr{H}.$ Then
\begin{equation*}
\{ \ran (M_{S, T})\}^\perp=G(S)^\perp\cap\{ F_{\mathscr{K}, \mathscr{H}}G(T)\}^\perp.
\end{equation*}
\end{proposition}

\begin{proof}
For $u\in \mathscr{H}, v\in\mathscr{K}$ the property
\begin{equation*}
(u, v)\in \{ \ran (M_{S, T})\}^\perp
\end{equation*}
is equivalent to the following identity
\begin{equation*}
\langle u, h-Tk \rangle + \langle v, Sh+k \rangle=0, \quad h\in \dom S,\ k \in \dom T.
\end{equation*}
One can take, on one hand, $h\in\dom S$ (arbitrary) and $k=0$ and, on the other hand, $h=0$ and $k\in\dom T$ (arbitrary) to see that the equalities
\begin{equation*}
\langle u, h\rangle + \langle v, Sh\rangle=\langle u, Tk\rangle - \langle v, k\rangle=0,\quad h\in\dom S,\ k\in\dom T
\end{equation*}
are also satisfied, i.e., $(u, v)\in G(S)^\perp\cap \{ F_{\mathscr{K}, \mathscr{H}}G(T)\}^\perp.$
\end{proof}

In our approach the situation when $M_{S, T}$ has full or at least dense range is discussed. 

We start with some necessary conditions for the surjectivity of $M_{S, T}.$

\begin{remark}\label{r1p}
Assume that $\ran(M_{S, T})=\mathscr{H}\times\mathscr{K}$ for two given linear operators $S$ (acting between $\mathscr{H}$ and $\mathscr{K}$) and $T$ (acting between $\mathscr{K}$ and $\mathscr{H}$). Then 
\begin{equation*}
\ran(1_{\dom(ST)}+ST)=\mathscr{K} \mbox{ and } \ran(1_{\dom(TS)}+TS)=\mathscr{H}.
\end{equation*}
Indeed, for every given $h\in\mathscr{H}$ one can find $h'\in\dom S$ and $k'\in\dom T$ such that
\begin{equation*}
h=h'-Tk' \mbox{ and } k'=-Sh'.
\end{equation*}
It follows that $h'\in\dom(TS)$ and $h=h'+TSh'.$ Hence $\ran(1_{\dom(TS)}+TS)=\mathscr{H}$ and, similarly, $\ran(1_{\dom(ST)}+ST)=\mathscr{K}.$ If $S$ and $T$ are closed and densely defined, and the resolvent sets of the operators $ST$ and $TS$ are non-empty it was noted by F. Philipp, A. Ran and M. Wojtylak in \cite{prw} that the surjectivity of $1_{\dom(ST)}+ST$ is equivalent with the surjectivity of $1_{\dom(TS)}+TS.$\qed
\end{remark}

In particular, the case $S=T^*$ is presented in the following proposition.

\begin{proposition}\label{p2}
Let $T$ be a densely defined linear operator between $\mathscr{K}$ and $\mathscr{H}.$ Then
\begin{itemize}
\item[$(a)$] $\overline{\ran (M_{T^*, T})}=\mathscr{H}\times\mathscr{K};$
\item[$(b)$] $\ran (M_{T^*, T})=\mathscr{H}\times\mathscr{K}$ if and only if $T$ is closed.
\end{itemize}
\end{proposition}

\begin{proof}
$(a)$ We firstly note that, according to \eqref{eq2} (written for $T$ instead of $S$),
\begin{equation}\label{eq2p}
\mathscr{H}\times\mathscr{K}=F_{\mathscr{K}, \mathscr{H}}(\mathscr{K}\times\mathscr{H})=F_{\mathscr{K}, \mathscr{H}} \overline{G(T)}\oplus G(T^*).
\end{equation}
Then, by Proposition \ref{p1}, 
\begin{equation*}
\ran (M_{T^*, T})^\perp=G(T^*)^\perp\cap \{ F_{\mathscr{K}, \mathscr{H}}G(T)\}^\perp=G(T^*)^\perp\cap G(T^*)=\{ 0 \},
\end{equation*}
as required.

$(b)$ If $M_{T^*, T}$ has full range then for a given $(k, h)\in\overline{G(T)}$ and for $h'\in\dom T^*$ one can find $h''\in\dom T^*$ and $k''\in\dom T$ such that 
\begin{equation*}
(h+h', T^*h'-k)=(h''-Tk'', T^*h''+k'').
\end{equation*}
This formula can be rewritten in equivalent form as 
\begin{equation*}
(h', T^*h')+F_{\mathscr{K}, \mathscr{H}}(k, h)=(h'', T^*h'')+F_{\mathscr{K}, \mathscr{H}}(-k'', T(-k'')).
\end{equation*}
Since the sum in the right side of \eqref{eq2p} is orthogonal (hence also direct) we deduce that $(k, h)=(-k'', T(-k''))\in G(T).$ Consequently, $T$ is a closed operator.

The converse implication follows by \cite[Theorem 8/page 71]{biso}.
\end{proof}

By Proposition \ref{p1}, $M_{S, T}$ has dense range if and only if
\begin{equation}\label{eq3}
\overline{\overline{G(S)}+F_{\mathscr{K}, \mathscr{H}}\overline{G(T)}}=\mathscr{H}\times\mathscr{K}.
\end{equation}
If $S$ and $T$ are densely defined it is possible to obtain a spectral characterization of this property.
To this aim, let $h\in\mathscr{H}$ and $k\in\mathscr{K}$ and observe that the property
\begin{equation*}
(h, k)\in G(S)^\perp\cap\{ F_{\mathscr{K}, \mathscr{H}} G(T) \}^\perp
\end{equation*}
is equivalent, in view of \eqref{eq2} and \eqref{eq2p}, with the following relations
\begin{equation*}
(k, -h)\in G(S^*) \mbox{ and } (h, k)\in G(T^*).
\end{equation*}
In other words, the statements $h\in\dom T^*, k\in \dom S^*, -h=S^*k$ and $k=T^*h$ are also satisfied. This implies that $h\in \dom(S^*T^*)$ and
\begin{equation*}
-h=S^*k=S^*T^*h.
\end{equation*}
Hence $h\in\ker(1_{\dom(S^*T^*)}+S^*T^*)$ and $(h, k)\in G(T^*|_{\ker(1_{\dom(S^*T^*)}+S^*T^*)}).$ 
Conversely, if $h\in \dom(S^*T^*)$ and $-h=S^*T^*h$ then $h\in\dom T^*, k:=T^*h\in\dom S^*$ and $S^*k=-h,$ so $(h,k)\in\{\ran(M_{S,T})\}^\perp$.
It follows that
\begin{equation*}
\{ \ran(M_{S, T})\}^\perp=G(T^*|_{\ker(1_{\dom(S^*T^*)}+S^*T^*)})
\end{equation*}
and, similarly,
\begin{equation*}
\{ \ran(M_{S, T})\}^\perp=F_{\mathscr{K}, \mathscr{H}}G(S^*|_{\ker(1_{\dom(T^*S^*)}+T^*S^*)}).
\end{equation*}

We deduce, in particular, the following

\begin{proposition}\label{p3}
Let $S$ and $T$ be densely defined linear operators acting between $\mathscr{H}$ and $\mathscr{K}$, respectively $\mathscr{K}$ and $\mathscr{H}.$ The following statements are equivalent:
\begin{itemize}
\item[$(i)$] $\overline{\ran(M_{S, T})}=\mathscr{H}\times\mathscr{K};$
\item[$(ii)$] $-1\not\in \sigma_p(S^*T^*);$
\item[$(iii)$] $-1\not\in\sigma_p(T^*S^*).$
\end{itemize}
\end{proposition}

\begin{corollary}\label{c4}
Let $S$ and $T$ be closed and densely defined linear operators between $\mathscr{H}$ and $\mathscr{K}$, respectively $\mathscr{K}$ and $\mathscr{H}$ such the resolvent sets\footnote{The resolvent set $\rho(S)$ of an unbounded operator $S$ on $\mathcal{H}$ is defined 
as the set of $\lambda\in\mathbb{C}$ for which the operator $S-\lambda$ is bijective from $\dom S$ onto $\mathcal{H}$ and the inverse operator $(S-\lambda)^{-1}$ is bounded on $\mathcal{H}$} of the operators $ST$ and $TS$ are non-empty. The following statements are equivalent:
\begin{itemize}
\item[$(i)$] $\overline{\ran (M_{S, T})}=\mathscr{H}\times\mathscr{K};$
\item[$(ii)$] $\overline{\ran(1_{\dom(ST)}+ST)}=\mathscr{K};$
\item[$(iii)$] $\overline{\ran(1_{\dom(TS)}+TS)}=\mathscr{H}.$
\end{itemize}
\end{corollary}

\begin{proof}
It was proved in \cite{hkm} and \cite{HM} that, under these hypotheses, $ST$ and $TS$ are densely defined, $(ST)^*=T^*S^*$ and $(TS)^*=S^*T^*.$ Also, by \cite{prw}, $1_{\dom(ST)}+ST$ has dense range if and only if $1_{\dom(TS)}+TS$ has dense range. In addition, 
\begin{eqnarray*}
\ker(1_{\dom(T^*S^*)}+T^*S^*)&=&\ran(1_{\dom(ST)}+ST)^\perp\\
&=&\ran(1_{\dom(TS)}+TS)^\perp\\
&=&\ker(1_{\dom(S^*T^*)}+S^*T^*).
\end{eqnarray*}
The conclusion follows by Proposition \ref{p3}.
\end{proof}

\begin{remark}\label{r5}
Similar results can be obtained if we replace $M_{S, T}$ by an operator matrix of the form
\begin{equation*}
M_{S, T} (\lambda):=\begin{pmatrix}
\lambda & -T\\ 
S & \lambda
\end{pmatrix}
\end{equation*}
defined for every given real or complex number $\lambda\ne 0.$ More precisely, since $M_{S, T}(\lambda)=\lambda M_{\frac{1}{\lambda}S, \frac{1}{\lambda}T},$ 
\begin{equation*}
\{ \ran (M_{S, T}(\lambda))\}^\perp=G\left(\frac{1}{\lambda}S\right)^\perp\cap\left\{ F_{\mathscr{K}, \mathscr{H}}G\left(\frac{1}{\lambda}T\right)\right\}^\perp
\end{equation*}
and, in the particular case when $S$ and $T$ are densely defined,
\begin{equation*}
\{ \ran (M_{S, T}(\lambda))\}^\perp=G\left(\frac{1}{\bar{\lambda}}T^*|_{\ker(\bar{\lambda}^2+S^*T^*)}\right)=F_{\mathscr{K}, \mathscr{H}}G\left(\frac{1}{\bar{\lambda}}S^*|_{\ker(\bar{\lambda}^2+T^*S^*)}\right).
\end{equation*}
It follows that $M_{S, T}(\lambda)$ has dense range if and only if $-\bar{\lambda}^2\in\sigma_p(S^*T^*)$ (respectively, $-\bar{\lambda}^2\in\sigma_p(T^*S^*)$). If, moreover, $S$ and $T$ are closed and the resolvent sets of the operators $ST$ and $TS$ are non-empty then $M_{S, T}(\lambda)$ has dense range if and only if $\lambda^2+ST$ (respectively, $\lambda^2+TS$) has dense range.\qed
\end{remark}

\begin{remark}\label{r27}
In the following sections it is sometimes useful to express $S\wedge T$ in terms of similar properties for operator matrices of the form $M_{S,T}(\lambda)$ with $\lambda\in\C$. More precisely, the following statements are equivalent:
\begin{itemize}
\item[$(i)$] $S\wedge T$;
\item[$(ii)$] $M_{S,T}(\lambda)\wedge M_{-S,-T}(\bar\lambda)$
for a certain (and also for all) $\lambda\in\C$.
\end{itemize}
Indeed, condition $(ii)$ takes the form
\begin{multline*}
\langle \lambda h-Tk, h'\rangle+ \langle Sh +\lambda k, k'\rangle=\langle h, \bar{\lambda} h'+ Tk'\rangle+\langle k, -Sh'+ \bar{\lambda} k'\rangle, \\ h, h'\in \dom S, \ k, k'\in \dom T.
\end{multline*}

This equality can be immediately rewritten as
\begin{multline*}
\langle Sh, k'\rangle-\langle h, Tk'\rangle=\langle Tk, h'\rangle-\langle k, Sh'\rangle,\quad
h, h'\in \dom S, \ k, k'\in\dom T,
\end{multline*}
that is, as $S\wedge T.$\qed
\end{remark}

\section{Closed operators}

The main tools to be used in this section are given by the following auxiliary result.

\begin{lemma}\label{l6}
Let $S$ and $T$ be linear operators acting between $\mathscr{H}$ and $\mathscr{K}$, respectively $\mathscr{K}$ and $\mathscr{H}.$ The following statements are equivalent:

\noindent
$(i)$
\begin{itemize}
\item[$(a)$] $\ran S=\mathscr{K};$
\item[$(b)$] $\overline{\ran T}=\mathscr{H};$
\item[$(c)$] $S \wedge T;$
\end{itemize}

\noindent
$(ii)$
\begin{itemize}
\item[$(a)$] $S$ is bounded from below;
\item[$(b)$] $\overline{\ran S}=\mathscr{K};$
\item[$(c)$] $\overline{\dom T}=\mathscr{K};$
\item[$(d)$] $S=T^*.$
\end{itemize}
\end{lemma}

\begin{proof}
$(i)\Rightarrow(ii).$ Let $k\in\{ \dom T\}^\perp$ and $h\in\dom S$ such that $k=Sh.$ Then, for each $k'\in\dom T,$
\begin{equation*}
\langle h, Tk'\rangle=\langle Sh, k'\rangle=0,
\end{equation*}
so $h\in\{ \ran T\}^\perp=\{ 0\}.$ It follows that $k=0$ and, consequently, $T$ is densely defined. Moreover, by $(c)$, $S\subseteq T^*.$ To obtain the equality it remains to show that $S$ and $T^*$ have equal domains. To this aim, let $h\in\dom T^*$ and consider $h'\in\dom S$ such that $T^*h=Sh'.$ Then $h-h'\in\ker T^*=\{ \ran T\}^\perp=\{ 0\}.$ Hence $h=h'\in\dom S,$ as claimed. We deduce that $S=T^*$ is a closed operator. It is also injective as 
\begin{equation*}
\ker S=\{ \ran T\}^\perp=\{ 0 \}.
\end{equation*}
Moreover,
\begin{equation*}
\dom S^{-1}=\ran S=\mathscr{K}
\end{equation*}
and, by the closed graph theorem, $S^{-1}$ is bounded.

$(ii)\Rightarrow(i).$ $\ S$ is surjective since, on one hand, $S$ is bounded from below and closed, hence its range is closed and, on the other hand, $S$ has dense range. In addition,
\begin{equation*}
\{\ran T\}^\perp=\ker S=\{ 0\},
\end{equation*}
which shows that the range of $T$ is also dense. Finally, $S=T^*$ obviously implies that $S$ and $T$ are adjoint to each other. 
\end{proof}

We deduce, in particular, new characterizations for skewadjoint, respectively selfadjoint operators.

\begin{corollary}\label{c7}
Let $S$ be a linear operator acting on the Hilbert space $\mathscr{H}.$ The following statements are equivalent:

\noindent
$(i)$ $\ S$ is (densely defined and) skewadjoint;

\noindent
$(ii)$
\begin{itemize}
\item[$(a)$] $S\wedge (-S);$
\item[$(b)$] $\ran (\lambda +S)=\mathscr{H};$
\item[$(c)$] $\overline{\ran (\lambda -S)}=\mathscr{H},$
\end{itemize}
for a certain (and also for all) $\lambda\in\mathbb{R}, \ \lambda\ne 0.$
\end{corollary}

\begin{proof}
$(i)\Rightarrow(ii).$ We firstly observe that, for $h\in \dom S,$
\begin{equation*}
\| (\lambda\pm S)h\|^2=| \lambda |^2\| h \|^2+\| Sh\|^2
\end{equation*}
since, obviously, $S$ and $-S$ are adjoint to each other (i.e., $\Re \langle Sh, h \rangle=0$ for $h\in\dom S$). This shows that the ranges of $\lambda\pm S$ are closed and, in addition,
\begin{equation*}
\ran (\lambda\pm S)^\perp=\ker (\lambda\mp S)=\{ 0 \}.
\end{equation*}
Hence $\ran (\lambda\pm S)=\mathscr{H}.$

$(ii)\Rightarrow(i).$ It is a simple application of Lemma \ref{l6} for the pair of operators $(\lambda + S, \lambda - S).$
\end{proof}

We replace $S$ with $iS$ to obtain a revised, improved version of the well-known von Neumann theorem (\cite{neu30}).

\begin{corollary}\label{c8}
Let $S$ be a linear operator acting on the complex Hilbert space $\mathscr{H}.$ The following statements are equivalent:

\noindent
$(i)$ $\ S$ is (densely defined and) selfadjoint;

\noindent
$(ii)$
\begin{itemize}
\item[$(a)$] $S\wedge S;$
\item[$(b)$] $\ran (\lambda i + S)=\mathscr{H};$
\item[$(c)$] $\overline{\ran (\lambda i - S)}=\mathscr{H},$
\end{itemize}
for a certain (and also for all) $\lambda\in\mathbb{R}, \ \lambda\ne 0.$
\end{corollary}

We are now in position to formulate the main result of this section which involves the operator matrix $M_{S, T}(\lambda).$

\begin{theorem}\label{t9}
Let $S$ and $T$ be linear operators acting between $\mathscr{H}$ and $\mathscr{K}$, respectively $\mathscr{K}$ and $\mathscr{H}.$ The following statements are equivalent:

\noindent
$(i)$
\begin{itemize}
\item[$(a)$] $S$ and $T$ are densely defined;
\item[$(b)$] $S^*=T$ and $S=T^*;$
\end{itemize}

\noindent
$(ii)$
\begin{itemize}
\item[$(a)$] $S\wedge T;$
\item[$(b)$] $\ran M_{S, T}(\lambda)=\mathscr{H}\times\mathscr{K}$ for a certain (and also for all) $\lambda\in\mathbb{R}, \ \lambda\ne 0;$
\end{itemize}

\noindent
$(iii)$
\begin{itemize}
\item[$(a)$] $S$ and $T$ are closed and densely defined;
\item[$(b)$] $S\wedge T;$
\item[$(c)$] $\ran (\lambda^2+ST)=\mathscr{K}$ and $\ran (\lambda^2+TS)=\mathscr{H}$ for a certain (and also for all) $\lambda\in\mathbb{R}, \ \lambda\ne 0.$
\end{itemize}

If $ST$ and $TS$ are also closed then the condition $(iii)\ (c)$ can be replaced by

\begin{itemize}
\item[$(c')$] $\ \overline{\ran (\lambda^2+ST)}=\mathscr{K}$ and $\overline{\ran (\lambda^2+TS)}=\mathscr{H}$ for a certain (and also for all) $\lambda\in\mathbb{R}, \ \lambda\ne 0.$
\end{itemize}
\end{theorem}

\begin{proof}
Observe that in all of the statements $(i) - (iii)$ the condition $S \wedge T$ holds true. Therefore, for each $h\in\dom S$ and $k\in \dom T,$ we have the following identity
\begin{eqnarray*}
\left\| M_{\pm(S, T)}(\lambda)\begin{pmatrix} h \\ k \end{pmatrix} \right\|^2&=&\| \lambda h\mp Tk\|^2+\| Sh\pm \lambda k\|^2\\ &=&| \lambda |^2\| h\|^2+| \lambda |^2\| k \|^2+ \| Sh\|^2+\| Tk\|^2,
\end{eqnarray*}
which shows that the operator matrices $M_{S, T}(\lambda)$ and $M_{-S, -T}(\lambda)$ are both bounded from below.

$(i)\Leftrightarrow(ii)$. In view of the observation above the equivalence follows by an application of Lemma \ref{l6} for the pair of operators $(M_{S, T}(\lambda), M_{-S, -T}(\lambda)).$ We want to emphasize, in this sense, that $\dom M_{S, T}(\lambda)=\dom S\times \dom T$ is dense if and only if $S$ and $T$ are densely defined and, in this case, $M_{S, T}(\lambda)=M_{-S, -T}(\lambda)^*\  (=M_{T^*, S^*}(\lambda))$ if and only if $S=T^*$ and $T=S^*.$ If $(i)$ holds true then $T$ is closed so, by Proposition \ref{p2}, $\ran M_{S, T}(\lambda)=\mathscr{H}\times\mathscr{K}.$ Finally, by Remark \ref{r27}, the condition $M_{S, T}(\lambda)\wedge M_{-S, -T}(\lambda)$ is equivalent to $S\wedge T$.

$(ii)\Leftrightarrow(iii)$. If $M_{S, T}(\lambda)$ is surjective then, by Remark \ref{r1p}, $\lambda^2+ ST$ and $\lambda^2+TS$ are also surjective. Hence $(ii)$ implies $(iii).$ 
Conversely, let us suppose that $S$ and $T$ are closed, densely defined and adjoint to each other. Then $\lambda^2+ST$ is bounded from below since, for all $k\in\dom(ST),$
\begin{equation*}
\| \lambda^2 k+STk\|^2=\lambda^4\| k\|^2+2\lambda^2\| Tk \|^2+\| STk\|^2.
\end{equation*}
We deduce, in view of the fact that $\lambda^2+ST$ has full range, that $-\lambda^2$ is in the resolvent set of $ST.$ The same is true when $ST$ is replaced by $TS.$ The hypotheses of Corollary \ref{c4} are satisfied. It follows that the operator matrix $M_{S, T}(\lambda)$ has dense range. 
Moreover, $M_{S, T}(\lambda)$ is closed (as $S$ and $T$ are closed) and bounded from below, so it also has closed range. Consequently, $\ran M_{S, T}=\mathscr{H}\times \mathscr{K}$ and this shows that $(iii)$ implies $(ii).$

Finally, if $ST$ and $TS$ are closed operators and $S,\ T$ satisfy conditions $(a),\ (b)$ (of $(iii)$) and $(c')$
then, as $\lambda^2+ST$ and $\lambda^2+TS$ are bounded from below, they must have closed (hence also full) ranges. This shows that $(iii)(c)$ can be replaced by $(c')$.
\end{proof}

New conditions on a linear operator, characterizing its skewadjointness, respectively selfadjointness, are obtained.

\begin{corollary}\label{c10}
Let $S$ be a linear operator on $\mathscr{H}.$ The following statements are equivalent:

\noindent
$(i)$ $\ S$ is (densely defined and) skewadjoint;

\noindent
$(ii)$
\begin{itemize}
\item[$(a)$] $S\wedge (-S);$
\item[$(b)$] $\ran M_{S, -S}(\lambda)=\mathscr{H}\times\mathscr{H}$ for a certain (and also for all) $\lambda\in\mathbb{R},\ \lambda\ne 0.$
\end{itemize}

\noindent
$(iii)$
\begin{itemize}
\item[$(a)$] $S$ is closed and densely defined;
\item[$(b)$] $S\wedge (-S);$
\item[$(c)$] $\ran (\lambda^2-S^2)=\mathscr{H}$ for a certain (and also for all) $\lambda\in\mathbb{R},\ \lambda\ne 0.$
\end{itemize}

If $S^2$ is also closed then the condition $(iii)\ (c)$ can be replaced by

\begin{itemize}
\item[$(c')$] $\overline{\ran (\lambda^2-S^2)}=\mathscr{H}$ for a certain (and also for all) $\lambda\in\mathbb{R},\ \lambda\ne 0.$ 
\end{itemize}
\end{corollary}

\begin{proof}
It follows by Theorem \ref{t9}, the case $T=-S.$
\end{proof}

\begin{corollary}\label{c11}
Let $S$ be a linear operator on $\mathscr{H}.$ The following statements are equivalent:

\noindent
$(i)$ $\ S$ is (densely defined and) selfadjoint;

\noindent
$(ii)$
\begin{itemize}
\item[$(a)$] $S\wedge S;$
\item[$(b)$] $\ran M_{S, S}(\lambda)=\mathscr{H}\times\mathscr{H}$ for a certain (and also for all) $\lambda\in\mathbb{R},\ \lambda\ne 0.$
\end{itemize}

\noindent
$(iii)$
\begin{itemize}
\item[$(a)$] $S$ is closed and densely defined;
\item[$(b)$] $S\wedge S;$
\item[$(c)$] $\ran (\lambda^2+S^2)=\mathscr{H}$ for a certain (and also for all) $\lambda\in\mathbb{R},\ \lambda\ne 0.$
\end{itemize}

If $S^2$ is also closed then the condition $(iii)\ (c)$ can be replaced by

\begin{itemize}
\item[$(c')$] $\overline{\ran (\lambda^2+S^2)}=\mathscr{H}$ for a certain (and also for all) $\lambda\in\mathbb{R},\ \lambda\ne 0.$ 
\end{itemize}
\end{corollary}

\begin{proof}
It follows by Theorem \ref{t9}, the case $T=S.$
\end{proof}

In the last part of this section we provide two conditions which characterize (bounded) orthogonal projections.

\begin{corollary}\label{c12}
Let $P$ be a linear operator acting on the Hilbert space $\mathscr{H}.$ The following statements are equivalent:

\noindent
$(i)$ $\ P$ is a (bounded) selfadjoint projection;

\noindent
$(ii)$
\begin{itemize}
\item[$(a)$] $P=P^2;$
\item[$(b)$] $P\wedge P;$
\item[$(c)$] $\ran (\lambda^2+P)=\mathscr{H}$ for a certain (and also for all) $\lambda\in\mathbb{R},\ \lambda\ne 0.$
\end{itemize}

\noindent
$(iii)$
\begin{itemize}
\item[$(a)$] $P=P^2;$
\item[$(b)$] $P\wedge P;$
\item[$(c)$] $\ran M_{P, P}(\lambda)=\mathscr{H}\times\mathscr{H}$ for a certain (and also for all) $\lambda\in\mathbb{R},\ \lambda\ne 0.$
\end{itemize}
\end{corollary}

\begin{proof}
The implication $(i)\Rightarrow(ii)$ is given by Corollary \ref{c11} ($(i)\Rightarrow(iii)$).

$(ii)\Rightarrow(iii).$ Let $h\in\mathscr{H}$ and $h'\in\dom P$ such that $h=\lambda^2 h'+Ph'.$ Since $\ran P\subseteq \dom P$ we deduce that $h\in \dom P.$ Hence $\dom P=\mathscr{H}.$

Let $(u, v)\in\mathscr{H}\times\mathscr{H},\ h=u+\frac{Pv-Pu}{2}$ and $k=v-\frac{Pu+Pv}{2}.$ Simple computations show that 
\begin{equation*}
h-Pk=u \mbox{ and } Ph+k=v,
\end{equation*}
that is, $\ran M_{P, P}=\mathscr{H}\times\mathscr{H}.$

$(iii)\Rightarrow(i).$ If $M_{P, P}(\lambda)$ is surjective then, by Remark \ref{r1p}, $\lambda^2+P$ is also surjective. As in the proof of the previous implication $P$ is defined on the whole space $\mathscr{H}.$ The conclusion follows immediately by Corollary \ref{c11} ($(ii)\Rightarrow(i)$).
\end{proof}

\section{The case: $S$ closable and $T$ densely defined}

This section studies the condition $S\wedge T$ in a more general context, in which $S$ is a closable operator and $T$ is densely defined.

\begin{proposition}\label{p13}
Let $\lambda\in\mathbb{R}, \lambda\ne 0$ and $S, T$ be linear operators between Hilbert space $\mathscr{H}$ and $\mathscr{K}$, respectively $\mathscr{K}$ and $\mathscr{H}.$ The following statements are equivalent:

\noindent
$(i)$ $\ \overline{G\Bigl( \frac{1}{\lambda}S\Bigr)}\oplus F_{\mathscr{K}, \mathscr{H}}
\overline{G\Bigl( \frac{1}{\lambda}T\Bigr)}=\mathscr{H}\times\mathscr{K};$

\noindent
$(ii)$
\begin{itemize}
\item[$(a)$] $S\wedge T;$
\item[$(b)$] $\overline{\ran M_{S, T}(\lambda)}=\mathscr{H}\times\mathscr{K}.$
\end{itemize}
\end{proposition}

\begin{proof}
$G\bigl( \frac{1}{\lambda}S\bigr)$ and $F_{\mathscr{K}, \mathscr{H}}G\bigl( \frac{1}{\lambda}T\bigr)$ are orthogonal if and only if $S\wedge T$ holds true and, in view of \eqref{eq3}, we have that $(ii)\ (b)$ is equivalent to the following identity
\begin{equation*}
\overline{G\Bigl( \frac{1}{\lambda}S\Bigr)} + F_{\mathscr{K}, \mathscr{H}}\overline{G\Bigl( \frac{1}{\lambda}T\Bigr)}=\mathscr{H}\times\mathscr{K}. 
\end{equation*}

The proof is complete.
\end{proof}

We are in a position to prove the main result in this section.

\begin{theorem}\label{t14}
Let the linear operators $S, T$ which act between Hilbert spaces $\mathscr{H}$ and $\mathscr{K},$ respectively $\mathscr{K}$ and $\mathscr{H}$ satisfy the property
\begin{equation}\label{eq4}
\overline{G\Bigl( \frac{1}{\lambda}S\Bigr)}\oplus F_{\mathscr{K}, \mathscr{H}}
\overline{G\Bigl( \frac{1}{\lambda}T\Bigr)}=\mathscr{H}\times\mathscr{K}
\end{equation}
for a certain given $\lambda\in\mathbb{R},\ \lambda\ne 0.$

Then the following two statements are equivalent:

$(i)$ $\ S$ is closable;

$(ii)$ $\ T$ is densely defined.

If either of the conditions $(i)$ or $(ii)$ is satisfied then the equality $\bar{S}=T^*$ holds true.

Conversely, if $\bar{S}=T^*$ then \eqref{eq4} holds true for every $\lambda\in\mathbb{R},\ \lambda\ne 0.$
\end{theorem}

\begin{proof}
The operator $S$ is closable if and only if, given $k\in\mathscr{K},\ \begin{pmatrix} 0 \\ k\end{pmatrix}\in\overline{G\bigl( \frac{1}{\lambda}S\bigr)}$ implies $k=0.$
But, under our assumption, $\begin{pmatrix} 0 \\ k\end{pmatrix}\in\overline{G\bigl( \frac{1}{\lambda}S\bigr)}$ means that, for each $k'\in \dom T,$
\begin{equation*}
0=\left\langle \begin{pmatrix} 0 \\ k\end{pmatrix}, \begin{pmatrix} \frac{1}{\lambda}Tk' \\ -k'\end{pmatrix}\right\rangle=-\langle k, k'\rangle
\end{equation*}
holds true. Equivalently, $k\in\{ \dom T\}^\perp.$ This last condition implies $k=0$ if and only if $\dom T$ is a dense subspace in $\mathscr{K}.$ But in this case $T^*$ exists and satisfies, by \eqref{eq2}, the following identity
\begin{equation*}
G\Bigl( \frac{1}{\lambda}T^*\Bigr)\oplus F_{\mathscr{K}, \mathscr{H}}
\overline{G\Bigl( \frac{1}{\lambda}T\Bigr)}=\mathscr{H}\times\mathscr{K}.
\end{equation*}
Therefore $G\bigl( \frac{1}{\lambda}T^*\bigr)=\overline{G\bigl( \frac{1}{\lambda}S\bigr)}=G\bigl( \frac{1}{\lambda}\bar{S}\bigr).$ In other words $T^*=\bar{S},$ as required.

Conversely, if $\bar{S}=T^*$ then, for every $\lambda\in\mathbb{R},\ \lambda\ne 0,$
\begin{eqnarray*}
\overline{G\Bigl( \frac{1}{\lambda}S\Bigr)}\oplus F_{\mathscr{K}, \mathscr{H}}
\overline{G\Bigl( \frac{1}{\lambda}T\Bigr)}=G\Bigl( \frac{1}{\lambda}T^*\Bigr)\oplus F_{\mathscr{K}, \mathscr{H}}
\overline{G\Bigl( \frac{1}{\lambda}T\Bigr)}=\mathscr{H}\times\mathscr{K}
\end{eqnarray*}
and the proof is complete.
\end{proof}

In the particular case when $S$ is densely defined and $T=S^*$ this theorem takes the form of the following classical result of J. von Neumann \cite{neu32}.
\begin{corollary}\label{c42p}
Let $S$ be a densely defined operator which acts between Hilbert spaces $\mathscr{H}$ and $\mathscr{K}$. Then $S$ is closable if and only if $S^*$ is densely defined in which case $\bar{S}=S^{**}$. 
\end{corollary}

In the most general context of Theorem \ref{t14} neither $S$ nor $T^*$ must be densely defined as we can observe from the following example.
\begin{example}\label{e42s}
Let $\mathscr{H}$ be a Hilbert space (infinite dimensional), $\mathscr{D}$ one of its dense proper subspaces and $\mathscr{X}\ne\{0\}$ a subspace of $\mathscr{H}$ satisfying $\mathscr{D}\cap\mathscr{X}=\{0\}$. We define a densely defined linear operator $T$ on $\mathscr{H}$ by
\begin{equation*}
\mathscr{H}\supseteq\dom T:=\mathscr{X}\dotplus\mathscr{D}\ni x+d\mapsto x+Ad\in\mathscr{H},
\end{equation*}
where $A$ is a given (bounded) positive operator on $\mathscr{H}$ with norm $\|A\|<1$ (the symbol ``$\dotplus$'' denotes a direct sum). An element $h\in\mathscr{H}$ is in the domain of $T^*$ if and only if
\begin{equation}\label{eq5p}
|\langle x+Ad,h\rangle|\le M\|x+d\|,\quad x\in\mathscr{X},\ d\in\mathscr{D}
\end{equation}
for a certain given $M>0$. By density any fixed $x\in\mathscr{X}$ is the limit of a sequence $(d_n)_{n\ge 0}$ of vectors in $\mathscr{D}$. The inequality \eqref{eq5p} written for $x$ and $-d_n,\ n\ge 0$ proves, by passing to limit, that $\langle x-Ax,h\rangle=0$. Hence any element $h\in\dom T^*$ is orthogonal to $(1-A)\mathscr{X}$. The converse is obviously true (one can take $M:=\|A\|\|h\|$), so
\begin{equation*}
\dom T^*=[(1-A)\mathscr{X}]^\perp.
\end{equation*}
Clearly $\dom T^*$ is closed and, since $1-A$ is invertible, $\dom T^*\ne\mathscr{H}$. We deduce that $T^*=A|_{\dom T^*}$ is not densely defined. We arrive to the same conclusion for the operator $S:=A|_{\dom S}$, where $\dom S$ is any given subspace of $\mathscr{H}$ which is dense in $\dom T^*$. One can easily observe that $S$ is closable and $\bar{S}=T^*$.\qed
\end{example} 

If $S$ and $T$ are adjoint to each other and $T$ is densely defined then, since $S\subseteq T^*$, $S$ is closable. The converse ($S$ closable implies $T$ densely defined) is generally not true without the additional assumption $\overline{\ran M_{S,T}(\lambda)}=\mathscr{H}\times\mathscr{K}\ (\lambda\in\R,\ \lambda\ne 0)$ as the following example shows.
\begin{example}\label{e42t}
Let $\mathscr{H}\ne\{0\}$ be a Hilbert space, $\mathscr{H}_0\ne\mathscr{H}$ one of its closed subspaces and $A$ a bounded selfadjoint operator on $\mathscr{H}$. Then $S=A|_{\mathscr{H}_0}:\mathscr{H}_0\to\mathscr{H}$ is obviously symmetric (i.e., $S$ and $T:=S$ are adjoint to each other), closed, but $T\ (=S)$ is not densely defined.\qed
\end{example}

In the example below we show that the condition $\overline{\ran M_{S,T}(\lambda)}=\mathscr{H}\times\mathscr{K}\ (\lambda\in\R,\ \lambda\ne 0)$ is essential for the equality $\bar{S}=T^*$ and cannot be obtained as a consequence of the other hypotheses of Theorem \ref{t14} ($S\wedge T$, $S$ is closable, $T$ is densely defined).
\begin{example}(\cite[Problem VIII.5]{rs80})
Let $\mathscr{H}=\ell^2_{\Z_+}(\C)$ be the Hilbert space of all square summable sequences of complex numbers and define the (unbounded) operator $S$ on $\mathscr{H}$ as
\begin{equation*}
\dom S:=\{(h_n)_{n\ge 0}\in\mathscr{H}\mid\text{for some }N\ge 0,\ \sum_{n=0}^Nh_n=0\text{ and }h_n=0\text{ if }n>N\}
\end{equation*}
and
\begin{equation*}
S(h_n)_{n\ge 0}:=(h'_n)_{n\ge 0},\ (h_n)_{n\ge 0}\in\dom S,
\end{equation*}
where $h'_n=
\begin{cases}
ih_n+2i\sum\limits_{m=0}^{n-1}h_m,& n\ge 1\\
ih_0,& n=0
\end{cases}
,\ n\ge 0$.

$\bullet$ \textit{$S$ is densely defined}.

Let us firstly observe that, for every $m\ge 0$, $e_m:=(\delta_{mp})_{p\ge 0}\in\overline{\dom S}$ being the limit of the sequence $(e_m-\frac{1}{n+1}\sum_{k=0}^ne_{m+k})_{n\ge 0}$ of vectors in $\dom S$. The conclusion follows by the remark that $\lim_{n\to\infty}\sum_{m=0}^nh_me_m=h$, for every $h=(h_n)_{n\ge 0}\in\mathscr{H}$.

$\bullet$ \textit{$S$ is symmetric}.

Easy computations show that
\begin{equation*}
\langle Sh,h\rangle=4\img \Bigl(\sum_{0\le i<j\le N-1}\bar{h_i}h_j\Bigr)
\end{equation*}
for every $h=(h_n)_{n\ge 0}\in\dom S$, where $N\ge 0$ is such that $\sum_{n=0}^Nh_n=0$ and $h_n=0$ for $n>N$.

$\bullet$ \textit{$S$ is not essentially selfadjoint}.

Since, for every $h\in\dom S$,
\begin{equation*}
\langle Sh,e_0\rangle=ih_0=\langle h,-ie_0\rangle
\end{equation*}
we deduce that $e_0\in\ker(S^*+i)$. Hence $\ker(S^*+i)\ne\{0\}$, so $S$ cannot be essentially selfadjoint.

We proved that $S$ and $T:=S$ are adjoint to each other, $S$ is closable and $T$ is densely defined, but the equality $\bar{S}=T^*$ is not satisfied.\qed
\end{example}

In view of Proposition \ref{p3} we can reformulate Theorem \ref{t14} as follows.

\begin{corollary}\label{c15}
Let $S$ and $T$ be densely defined linear operators between Hilbert spaces $\mathscr{H}$ and $\mathscr{K}$, respectively $\mathscr{K}$ and $\mathscr{H}.$ The following statements are equivalent:

\noindent
$(i)$ $\ S$ is closable and $\bar{S}=T^*;$

\noindent
$(ii)$
\begin{itemize}
\item[$(a)$] $S\wedge T;$
\item[$(b)$] $-\lambda^2\not\in\sigma_p(S^*T^*)$ for a certain (and also for all) $\lambda\in\mathbb{R},\ \lambda\ne 0;$
\end{itemize}

\noindent
$(iii)$
\begin{itemize}
\item[$(a)$] $S\wedge T;$
\item[$(b)$] $-\lambda^2\not\in\sigma_p(T^*S^*)$ for a certain (and also for all) $\lambda\in\mathbb{R},\ \lambda\ne 0.$
\end{itemize}
\end{corollary}

We also obtain, in particular, new conditions which characterize essential selfadjointness.

\begin{corollary}\label{c16}
Let $S$ be a linear operator acting on the Hilbert space $\mathscr{H}.$ The following statements are equivalent:

\noindent
$(i)$ $\ S$ is essentially selfadjoint;

\noindent
$(ii)$
\begin{itemize}
\item[$(a)$] $S$ is densely defined;
\item[$(b)$] $S\wedge S;$
\item[$(c)$] $\overline{\ran M_{S, S}(\lambda)}=\mathscr{H}\times\mathscr{H}$ for a certain (and also for all) $\lambda\in\mathbb{R},\ \lambda\ne 0;$
\end{itemize}

\noindent
$(iii)$
\begin{itemize}
\item[$(a)$] $S$ is closable;
\item[$(b)$] $S\wedge S;$
\item[$(c)$] $\overline{\ran M_{S, S}(\lambda)}=\mathscr{H}\times\mathscr{H}$ for a certain (and also for all) $\lambda\in\mathbb{R},\ \lambda\ne 0.$
\end{itemize}
\end{corollary}

\begin{proof}
It follows by Theorem \ref{t14}, the case $T=S.$
\end{proof}

\begin{corollary}\label{c17}
Let $S$ be a linear operator acting on the Hilbert space $\mathscr{H}.$ The following statements are equivalent:

\noindent
$(i)$ $\ S$ is essentially selfadjoint;

\noindent
$(ii)$
\begin{itemize}
\item[$(a)$] $S$ is densely defined;
\item[$(b)$] $S \wedge S$;
\item[$(c)$] $-\lambda^2\not\in\sigma_p(S^{*2})$ for a certain (and also for all) $\lambda\in\mathbb{R},\ \lambda\ne 0.$
\end{itemize}
\end{corollary}

\begin{proof}
It follows by Corollary \ref{c15}. 
\end{proof}

Other applications of Theorem \ref{t14} are characterizations for selfadjointness (the case $T=S$ with $S$ closed) and, respectively, skewadjointness (the case $T=-S$ with $S$ closed).

\begin{corollary}\label{c18}
Let $S$ be a linear operator acting on the Hilbert space $\mathscr{H}.$ The following statements are equivalent:

\noindent
$(i)$ $\ S$ is (densely defined and) selfadjoint;

\noindent
$(ii)$
\begin{itemize}
\item[$(a)$] $S$ is closed;
\item[$(b)$] $S \wedge S$;
\item[$(c)$] $\overline{\ran M_{S, S}(\lambda)}=\mathscr{H}\times\mathscr{H}$ for a certain (and also for all) $\lambda\in\mathbb{R},\ \lambda\ne 0.$
\end{itemize}
\end{corollary}

\begin{corollary}\label{c19}
Let $S$ be a linear operator acting on the Hilbert space $\mathscr{H}.$ The following statements are equivalent:

\noindent
$(i)$ $\ S$ is (densely defined and) skewadjoint;

\noindent
$(ii)$
\begin{itemize}
\item[$(a)$] $S$ is closed;
\item[$(b)$] $S \wedge (-S)$;
\item[$(c)$] $\overline{\ran M_{S, -S}(\lambda)}=\mathscr{H}\times\mathscr{H}$ for a certain (and also for all) $\lambda\in\mathbb{R},\ \lambda\ne 0.$
\end{itemize}
\end{corollary}

We introduce the following notations for a given linear operators $S$ acting between Hilbert spaces $\mathscr{H}$ and $\mathscr{K}:$
\begin{multline*}
N_S:=\{ h\in\mathscr{H} \mid \mbox{ there exists } \{ h_n \}_n\subseteq\dom S \mbox{ converging to } h\\ \mbox{ such that } \{ Sh_n\}_n \mbox{ converges and, for any such } \{ h_n \}_n, \ \lim_{n\to\infty} Sh_n=0\};
\end{multline*}
\begin{multline*}
Q_S:=\{ k\in\mathscr{K} \mid \mbox{ there exists a convergent sequence } \{ h_h\}_n\subseteq\dom S\\ \mbox{ such that } \lim_{n\to \infty} Sh_n=k\};
\end{multline*}
\begin{multline*}
R_S:=\{ h\in\mathscr{H} \mid \sup\{ | \langle h, h'\rangle | : h'\in\dom S \mbox{ and } \| Sh' \|\le 1\}<\infty\}.
\end{multline*}

\begin{remark}\label{r47p}
If $S$ is closable then a simple application of the formula $\overline{G(S)}=G(\bar{S})$ shows that
\begin{equation*}
N_S=\ker \bar{S} \mbox{ and } Q_S=\ran \bar{S}.
\end{equation*}
\qed
\end{remark}

\begin{theorem}\label{t20}
Let $S$ and $T$ be linear operators acting between Hilbert spaces $\mathscr{H}$ and $\mathscr{K},$ respectively $\mathscr{K}$and $\mathscr{H}.$ The following statements are equivalent:

\noindent
$(i)$
\begin{itemize}
\item[$(a)$] $S$ is closable;
\item[$(b)$] $T$ is densely defined;
\item[$(c)$] $\bar{S}=T^*;$
\end{itemize}

\noindent
$(ii)$
\begin{itemize}
\item[$(a)$] $S\wedge T;$
\item[$(b)$] $\{ \ran T\}^\perp=N_S;$
\item[$(c)$] $Q_S=R_T.$
\end{itemize}
\end{theorem}

\begin{proof}
$(i)\Rightarrow(ii).$ Let us suppose that $(i)$ holds true.
Then, by Remark \ref{r47p} and \eqref{eq0},
\begin{equation*}
\{ \ran T \}^\perp=\ker T^*=\ker \bar{S}=N_S
\end{equation*}
and
\begin{equation*}
R_T=\ran T^*=\ran \bar{S}=Q_S.
\end{equation*}

$(ii)\Rightarrow(i).$ We check first $(i)(b).$ To this aim let $k\in\{ \dom T\}^\perp\subseteq R_T.$
Then, according to $(ii)(c),$ there exists a sequence $\{ h_n\}_n\subseteq\dom S$ which is convergent to some $h\in\mathscr{H}$ and such that $k=\lim_{n\to \infty} Sh_n.$
It follows that, for $k'\in\dom T,$ the following equalities hold:
\begin{equation*}
0=\langle k', k\rangle=\lim_{n\to\infty} \langle k', Sh_n\rangle=\lim_{n\to\infty}\langle Tk', h_n\rangle=\langle Tk', h\rangle.
\end{equation*}
Therefore, $h\in\{ \ran T\}^\perp=N_S,$ which means that $k=\lim_{n\to\infty} Sh_n=0,$ as required.

Statement $(i)(a)$ follows by $(ii)(a)$ which, by $(i)(b)$, takes the form $S\subseteq T^*$. Since $T^*$ is closed we also have $\bar{S}\subseteq T^*$. At the same time we know that $R_T=\ran T^*$ (by \eqref{eq0}), $Q_S=\ran \bar{S}$ and $N_S=\ker \bar{S}$ (by Remark \ref{r47p}). Consequently, our hypotheses can be rewritten as $\ker T^*=\ker \bar{S}$ (condition $(ii)(b)$) and $\ran T^*=\ran\bar{S}$ (condition $(ii)(c)$). Finally, we get the condition $\bar{S}=T^*$ (statement $(i)(c)$) by Proposition \ref{p1.1}. 
\end{proof}

We present, as consequences, some other characterizations for essentially selfadjoint operators.

\begin{corollary}\label{c21}
Let $S$ be a linear operator acting on the Hilbert space $\mathscr{H}.$ The following statements are equivalent:

\noindent
$(i)$ $S$ is essentially selfadjoint;

\noindent
$(ii)$
\begin{itemize}
\item[$(a)$] $S \wedge S$;
\item[$(b)$] $\{ \ran S\}^\perp=N_S;$
\item[$(c)$] $Q_S=R_S.$
\end{itemize}
\end{corollary}

\begin{proof}
It follows by Theorem \ref{t20}, the case $T=S.$
\end{proof}

\begin{corollary}\label{c22}
Let $S$ be a closable linear operator acting on the Hilbert space $\mathscr{H}.$ The following statements are equivalent:

\noindent
$(i)$ $\ S$ is essentially selfadjoint;

\noindent
$(ii)$
\begin{itemize}
\item[$(a)$] $S \wedge S$;
\item[$(b)$] $(\ran S)^\perp=\ker \bar{S};$
\item[$(c)$] $R_S=\ran \bar{S}.$
\end{itemize}
\end{corollary}

\begin{proof}
It follows by Corollary \ref{c21} taking into account the fact that, under the assumption that $S$ is closable, $N_S=\ker \bar{S}$ and $Q_S=\ran \bar{S}.$
\end{proof}

\begin{corollary}\label{c23}
Let $S$ be a densely defined symmetric operator acting on the Hilbert space $\mathscr{H}.$ The following statements are equivalent:
\begin{itemize}
\item[$(i)$] $S$ is essentially selfadjoint;
\item[$(ii)$] $\ran S^*=\ran \bar{S}.$
\end{itemize}
\end{corollary}

\begin{proof}
In view of the previous corollary we only have to check that $\ran S^*=\ran\bar{S}$ implies $\ker S^*=\ker \bar{S}.$ This follows easily by the equalities:
\begin{equation*}
\ker S^*=\{ \ran S\}^\perp=\{ \ran \bar{S}\}^\perp=\{ \ran S^*\}^\perp=\ker S^{**}=\ker \bar{S}.
\end{equation*}
\end{proof}

\begin{acknowledgements}
The work of the first author has been supported by the Hungarian Scholarship Board.
\end{acknowledgements}

\end{document}